\documentclass[10pt]{amsart} 

\makeatletter
\@namedef{subjclassname@2010}{%
  \textup{2010} Mathematics Subject Classification}
\makeatother

\usepackage{url}
\usepackage{enumerate}
\usepackage{enumitem}
\usepackage[]{amsmath}
\usepackage{amsthm}
\usepackage{amssymb}
\usepackage[utf8]{inputenc}
\usepackage{verbatim}
\usepackage{tikz}
\theoremstyle{plain}
\newtheorem{theorem}{Theorem}[section]
\newtheorem{corollary}[theorem]{Corollary}
\newtheorem{proposition}[theorem]{Proposition}
\newtheorem{lemma}[theorem]{Lemma}

\theoremstyle{remark}
\newtheorem{remark}[theorem]{Remark}
\newtheorem{definition}[theorem]{Definition}

\newtheorem{example}[theorem]{Example}

\newcommand{\N}{\ensuremath{\mathbb{N}}}

\newcommand{\ii}{\ensuremath{\bold{a}}}
\newcommand{\I}{\ensuremath{\mathfrak{I}}}

\newcommand{\HH}{\mathcal{H}}
\newcommand{\rst}[1]{\ensuremath{{\mathbin |}\raisebox{-.5ex}{$#1$}}}  

\DeclareMathOperator{\diam}{diam}

\begin{document}
\title[ On $\left(\alpha_n \right)$-regular sets]{On $\left( \alpha_n \right)$-regular sets}
\date{\today}
\author[T. Ojala]{Tuomo Ojala}
\address{Department of Mathematics and Statistics\\
University of Jyv\"askyl\"a \\
P.O. Box 35 (MaD), FI-40014, Finland }
\email{tuomo.j.ojala@jyu.fi}
\begin{abstract}
We define $(\alpha_n)$ -regular sets in uniformly perfect metric spaces. 
This definition is quasisymmetrically invariant and the construction resembles generalized dyadic cubes in metric spaces.
For these sets we then determine the necessary and sufficient conditions to be fat (or thin). In addition we discuss
restrictions of doubling measures to these sets, and, in particular, give a sufficient condition to retain 
at least some of the restricted measures doubling on the set.
Our main result generalizes and extends analogous results that were previously known to hold on the real line.

\end{abstract}


\maketitle

\section{Introduction}
In this paper we discuss the size of sets in terms of doubling measures. A 
Borel regular (outer -) measure $\mu$ on a metric space $X$ is called \emph{doubling 
(with constant $C$)} if there exists a constant $C\geq 1$ such that
\begin{equation*}
  0 < \mu(B(x, 2 r)) \leq C \mu(B(x,r)) < \infty
\end{equation*}
for all balls $B(x,r) = \left\{ y \in X : d(x,y) < r \right\}$. 
Related to this, a metric space is called \emph{doubling}, if there exists a constant $N$ such that 
any ball of radius $r$ can be covered by $N$ balls of radius $\frac{1}{2}r$. A metric space that carries a doubling measure 
is doubling by a simple volume argument, and by well known results \cite{VolbergKonyagin1987} and \cite{LuukkainenSaksman1998} also the converse is 
true in complete spaces.

 We call a subset $E$ of a metric space $X$ \emph{thin} if it has zero measure with respect to all doubling measures and \emph{fat}
 if it has positive measure for all doubling measures of $X$. In literature fat sets have also been termed quasisymmetrically thick \cite{Heinonen}, thick \cite{Hanetal2009} and very fat \cite{BuHaMac,WangWenWen2013}. 
 Thin sets, on the other hand, have also been called quasisymmetrically null \cite{StaplesWard},
 null for doubling measures \cite{Wu} and very thin \cite{WangWenWen2013}. 

 Sets with nonempty interior are fat and countable sets with no isolated points are thin. Thus, the interest
 lies in uncountable sets without interior. Symmetric Cantor sets that are constructed from the unit interval by removing 
 the middle segment of relative length $\alpha_n$ from each line segment of the construction level $n$ offer 
 an example of sets 
 of this kind.
 For these Cantor sets there is a complete characterization in terms of the defining sequence. Let us denote 
 \begin{equation*}
   \ell^p := \left\{ (\alpha_{n})_{n=1}^{\infty} : 0 < \alpha_n < 1 \text{ and } \sum_{n=1}^{\infty} \alpha_n^{p} < \infty \right\}
 \end{equation*}
 and 
 \begin{equation*}
   \ell^0 := \bigcap_{0 < p} \ell^p, \quad  \ell^{\infty} := \bigcup_{0 < p} \ell^p.   
 \end{equation*}
 It is known that a symmetric Cantor set $C(\alpha_n)$ with defining sequence $(\alpha_n)$ is fat if and only if $(\alpha_n) \in \ell^0$
 and thin if and only if $(\alpha_n) \notin \ell^{\infty}$ (see \cite{Wu}, \cite{StaplesWard} and \cite{BuHaMac}).
 These results have also been generalized to \emph{nice $(\alpha_n)$-regular Cantor sets} and \emph{uniform Cantor sets} 
 of the real line
 (see \cite{CsoSuom}, \cite{Hanetal2009}, \cite{PengWen2011} and \cite{WangWenWen2013} for more precise definitions 
 and results).
 
 In more general metric spaces there are notions of $(\alpha_n)$ -thick (and -porous) sets. 
 For $(\alpha_n)$-thick sets it is known that 
 $(\alpha_n) \in \ell^0$ implies fatness and 
 for $(\alpha_n)$-porous sets $(\alpha_n) \notin \ell^{\infty}$ implies thinness (see \cite{CsoSuom} and \cite{OjRajSuo}).
 Our aim is now to give the missing parts of the characterizations of fatness/thinnes in terms of the defining sequence,
 for a natural class of Cantor type sets that generalize the symmetric Cantor sets of the real-line to more general spaces.
 For this, we will determine a class of sets that contain enough regularity for us to 
 work with. In addition, we will assume only a slight regularity of the space that we will work in.

 Throughout the whole paper we will be working in uniformly perfect metric spaces.
Recall that a metric space $X$ is called \emph{uniformly perfect (with constant $D$)},
if it is not a singleton and if there exists a constant $D \geq 1$
such that
\[
X \setminus B(x,r) \neq \emptyset \Rightarrow B(x,r) \setminus B(x,r/D) \neq \emptyset
\]
for all $x\in X$ and $r>0$. Recall also that in a uniformly perfect space, the diameter of a ball, $\diam B(x,r)$, 
and the radius $r$
are comparable: $r/D \leq \diam B(x,r) \leq 2 r$.

A metric measure space  $(X,\mu)$ is called \emph{Ahlfors ($q$-) regular} if there exist
constants $C\geq1$ and $0<q<\infty$ such that
\begin{equation}
  \frac{1}{C} r^q \leq \mu(B(x,r)) \leq C r^q
  \label{eqn:ahlforsdefinition}
\end{equation}
for all $r>0$ and $x \in X$. The measure $\mu$ is then also called \emph{Ahlfors regular} and it is comparable with the
$q$-dimensional Hausdorff 
measure.
In particular, the space $X$ has Hausdorff dimension $q$ and 
is uniformly perfect.
Later we will be denoting an Ahlfors regular measure by $\HH$, but we are not requiring that it has to be the 
Hausdorff measure.

Let us next recall the connection of doubling measures to quasisymmetric maps. A homeomorphism $f$
between two metric spaces $(X,d_X)$ and $(Y,d_Y)$ is \emph{($\eta$-) quasisymmetric} 
if there exists a homeomorphism 
$\eta \colon [0,\infty) \to [0,\infty)$ so that for any three distinct points $x,y,z \in X$ we have
\[
 \frac{d_Y(f(x),f(y))}{d_Y(f(x),f(z))} \le \eta\left(\frac{d_X(x,y)}{d_X(x,z)}\right).
 \]
 It follows from the definion that the quasisymmetric image of a uniformly perfect metric space is uniformly perfect,
 the inverse of a quasisymmetric function is quasisymmetric, 
 and the pullback of a doubling measure under a quasisymmetric function is doubling (see, for example, \cite{Heinonen}).


To state our main theorem we will give the following definition. 
The sets $Q_{n,j}$ below can be thought of as generalized cubes in a metric space. 
However, these are not necessarily dyadic cubes as discussed in \cite{HytKai,KaeSuoRaj,Christ}.
It is clear that the system of cubes that we describe is more general and does not possess all the
properties of dyadic cubes. The main issue is that unless the sequence $(\alpha_n)$ is constant,
the system in a sense does not have all the scales of dyadic cubes. 
This ``nonexistence'' of all the dyadic levels allows much more flexibility, as in Example \ref{ex:nondoubling}.

%

\begin{definition}
  Given a sequence $(\alpha_n),$ $0<\alpha_n<1, \; n \in \N$, we call a set $E$ \emph{$(\alpha_n)$-regular} if it 
  satisfies the following conditions:

  We have a collection of Borel sets $\left\{Q_{k,i}, k\in\N , i\in N_k \subset \N\right\} $
  and constants $d\leq 1,C_1 \geq 1,C_2 \geq 1$ such that 
\newcounter{toinen}
\begin{enumerate}
    \renewcommand{\theenumi}{(\Roman{enumi})}
    \renewcommand{\labelenumi}{\theenumi}
\item $X = \cup_{i\in N_k} Q_{k,i}$ (disjoint union) for every $k \in \N
  \label{list:cubespartition}
  $ 
\item  $Q_{k,i} \cap Q_{m,j} = \emptyset$ or $Q_{k,i} \subset Q_{m,j}$, for all $i \in N_k,j \in N_m$ and $k\geq m$, \label{list:cubesdisjointandchild}
\item For every $k\in \N$ and $i \in N_k$ 
  there exists a point $x_{k,i} \in X$ and radius $0 < r_{k,i} < \infty$ such that 
  \begin{equation*}
    B(x_{k,i}, r_{k,i}) \subset Q_{k,i} \subset B(x_{k,i},C_1 r_{k,i}),
    \label{eqn:cubesinballs}
  \end{equation*}
  \label{list:cubesinballs}
\item $ x_{n,i} \in Q_{n+1,j} \Rightarrow r_{n,i} \frac{1}{C_2} \alpha_n^{1/d} \leq r_{n+1,j}  \leq C_2 \alpha_n^d r_{n,i} $
  and $Q_{n,i} \setminus Q_{n+1,j} \neq \emptyset$,
  \label{list:centercubealpha}
\item For any $T>1$ there exists $1 \leq C_3 =C_3(T) < \infty$ such that
  \begin{equation*}
    Q_{n,i} \cap B(x_{n,j}, T r_{n,j})\neq \emptyset \Rightarrow \frac{1}{C_3} r_{n,j} \leq r_{n,i} \leq C_3 r_{n,j}.
  \end{equation*}
  \label{list:comparableradius}
\item $E:=\cap_n E_{n}$, where  $E_n:=\cup_{ Q_{n,i} } \left\{ Q_{n,i}\setminus Q_{n+1,j}: x_{n,i} \in Q_{n+1,j} \right\}$.
  \label{list:EisCubesminuscubes}
\end{enumerate}
\label{def:cubes}
\end{definition}

It should be noticed that this definition
includes Sierpinski carpets in higher dimensional Euclidean spaces.
Similarly to the Cantor sets $C(\alpha_n)$ that were mentioned earlier, Sierpinski 
carpets $S_{\ii}$ can be constructed according to the sequence of reciprocals of odd numbers
$\ii = (\frac{1}{a_{n}}), \; a_n \in \left\{ 3,5,7,\dots \right\}$.
The construction in question undergoes by dividing each of the level $n$ squares into $a_n^2$ subsquares 
in an obvious manner and removing 
the middle square of the level $n+1$ from each of the level $n$ squares.
In \cite{MacTysWil} these carpets were studied in connection with Poincar\'e inequalities.
Mackay et al. showed that $(S_{\ii},d,\mu)$ supports
a $p$-Poincar\'e inequality for $p>1$ if and only if $\ii \in \ell^2$. Here $d$ is Euclidean metric and
$\mu$ weak* limit of normalized Lebesgue measures on the pre-carpets, which in the case $\ii \in \ell^2$
is comparable to the restriction of the Lebesgue measure to $S_{\ii}$.

As a simple corollary of our main result (see Theorem \ref{thm:maintheorem})  
we see, in particular, 
that if $\ii \in \ell^{\infty}$, then $S_{\ii}$ will be of positive measure for some doubling measure of
the plane. We will prove that when restricted to the space $(S_\ii)$, these measures that we construct are also doubling 
as measures on the space $S_{\ii}$.
This is not true in the case of a general $(\alpha_n)$-regular set, and we give a counterexample of this. 
We also give sufficient condition for the construction of the set $E$, to guarantee that restrictions of our 
constructed measures will be
doubling measures on $E$.
Essentially what is required for this to be true is some quantified plumpness (see \cite{HytKai}) 
of the cubes in our construction. This discussion will be continued and made precise in the last section,
where also the above mentioned
example and results are provided.

Our $(\alpha_n)$-regular sets are slightly different from the nice $(\alpha_n)$-regular sets considered on the 
real line in \cite{CsoSuom}. Neither of the classes is contained in the other. 
One advantage of our definition (besides that it applies in very general metric spaces), 
is it's quasisymmetric invariance. We will prove the invariance in Lemma \ref{le:alpharegularQSinv}.

Let us next state our main theorem.

\begin{theorem}
 Let $E$ be an $(\alpha_n)$ -regular set  
 in a complete, doubling, uniformly perfect metric space $X$. Then $E$ is fat if and only if 
 $(\alpha_n) \in \ell^0$ and thin if and only if $(\alpha_n) \notin \ell^\infty$.
  \label{thm:maintheorem}
\end{theorem}
To prove Theorem \ref{thm:maintheorem}, we need to prove four implications:
\begin{enumerate}
    \renewcommand{\theenumi}{(\roman{enumi})}
    \renewcommand{\labelenumi}{\theenumi}
\item $(\alpha_n) \in \ell^0 \Rightarrow E$ is fat. \label{list:Efat}
\item $(\alpha_n) \notin \ell^{\infty} \Rightarrow E $ is thin.\label{list:Ethin}
\item $(\alpha_n) \notin \ell^0 \Rightarrow$ there exists a doubling measure $ \mu$ such that $\mu(E)=0$ \label{list:Enotfat}
\item $(\alpha_n) \in \ell^{\infty} \Rightarrow$ there exists a doubling measure $\nu$ such that $\nu(E) > 0$ 
  \label{list:Enotthin}
\end{enumerate}
Implications \ref{list:Efat} and \ref{list:Ethin} follow from \cite[Lemma 4.1]{CsoSuom} and \cite[Theorem 3.2]{OjRajSuo}
since an $(\alpha_n)$ -regular set as defined above is $\left( \frac{1}{C_2} \alpha_n^{1/d} \right)$-porous  
and $\left( \alpha_n^{d} \right)$-thick as defined in \cite{CsoSuom} and \cite{OjRajSuo}.
For the other two implications, the results of this type are only known on the real line. As already mentioned, 
the latest such results are
\cite[Theorem 1.1 and Corollary 2.3]{CsoSuom} and \cite[Theorem 1. and Theorem 2.]{Hanetal2009}, however,   
already in \cite{Heinonen} Heinonen asked to 
what extent the one dimensional results would have analogs in higher dimensions. Moreover, Wang et al. 
mention in \cite{WangWenWen2013} that, 
the question of which sets are thin or fat in higher dimensional Euclidean spaces is still open.
As a contribution to this question, they show that a product of $n$ uniform Cantor sets is fat 
if and only if each of the 
factors is fat and a product of $n$ sets is thin if and only if some of the factors is thin.


\begin{remark}[Existence of $\left( \alpha_n \right)$ -regular sets]
To motivate our result, we show that for a given sequence $(\alpha_n)$, where $\alpha_n \rightarrow 0$ as 
$n \rightarrow \infty$,
$(\alpha_n)$ -regular 
sets always exist in uniformly perfect metric space.
For this we use the generalized dyadic cubes constructed in \cite{KaeSuoRaj}.
Note that these exist in any doubling metric space.
The construction in \cite{KaeSuoRaj} yields for any $0<r<1/3$ a collection 
$\left\{Q_{i}^k, k\in\N , i\in N_k \subset \N\right\} $ of Borel sets with the following properties:

\begin{list}{\Roman{toinen}:}{\usecounter{toinen} }
\item $X = \cup_{i\in N_k} Q_{i}^k$, for every $k$, and $Q_{i}^{k} \cap Q_{j}^{k} = \emptyset$ for all $k$ and $i \neq j$,
\item $Q_{i}^n \cap Q_{j}^m = \emptyset$ or $Q_{i}^k \subset Q_{j}^m$, when $k,m \in \N, k\geq m, i\in \N_k, j\in N_m$,
\item for every $k\in \N$ and $i \in N_k$ there exists a point $x_{i}^k \in X$ such that 
  \begin{equation*}
    U(x_{k,i}, c r^k) \subset Q_{i}^k \subset B(x_{i}^k,C r^k),
    \label{eqn:cubesinballsKRS}
  \end{equation*}
  where $c=1/2- \frac{r}{1-r}$ and $C=\frac{1}{1-r}$.
\item $\left\{ x_{i}^k : i\in N_k \right\} \subset \left\{ x_{i}^{k+1} : i \in N_{k+1} \right\}$ for all $k \in \N$.
\end{list}

Let $X$ be a uniformly perfect metric space (with constant $D$) and let
sequence $(\alpha_n)$, where $\alpha_n \rightarrow 0$ as $n \rightarrow \infty$,  be given. 
For each $n \in \N$ we choose subcollections $\left\{ Q_{i}^{k_n}: i\in N_{k_n} \right\}$ corresponding to $k_n$ such
that $\frac{r^{k_{n+1}}}{r^{k_{n}}} \approx \alpha_n$, for all $n$. 
This means simply choosing $k_{n+1}:=\lceil \log_r \alpha_n + k_n \rceil$ inductively.
Now we are ready to set $Q_{n,i} := Q_{i}^{k_n}$, $N_n := N_{k_n}$, $x_{n,i}= :x_{k_n,i}$ and $r_{n,i} := \frac{1}{3} r^{k_n}$ 
with for example $r=1/7$ along with constants $d:=1$,$ C_1=6, C_2=3, C_3(T)=1$. 
To achieve the second condition in \ref{list:centercubealpha}, we renumber the indices starting from such large $N$ that 
$D C_1 C_2 \alpha_n < 1$ and $\diam(X)> 2 r_{n,j}$, for all $n\geq N$. Finally, we set
$E:=\cap_{n} E_n$,
where $E_n:= \cup_i \left\{ Q_{n,i} \setminus Q_{n+1,j} : x_{n,i} \in Q_{n+1,j} \right\}$ 
to get an $(\alpha_n)$ -regular set.
\end{remark}


\section{Quasisymmetric invariance}
Towards the proof of the implications \ref{list:Enotfat} and \ref{list:Enotthin} 
we first show that these statements are quasisymmetrically invariant.
Since doubling measures can be pushed forward (or pulled back) under quasisymmetric maps,
it is enough that we show the invariance of Definition \ref{def:cubes}.
\begin{lemma}
  If $X$ is a uniformly perfect metric space, $E \subset X$ is $(\alpha_n)$-regular and $f:X\to Y$ is $\eta$-quasisymmetric, then $f(E)$ is $(\alpha_n)$-regular.
  \label{le:alpharegularQSinv}
\end{lemma}

We will employ a couple of well known results about quasisymmetric maps. 
These can be found, for example, from \cite[Proposition 10.8 and Theorem 11.3]{Heinonen}.

\begin{lemma}
  If $f:X \to Y$ is $\eta$-quasisymmetric and if $A \subset B \subset X$ are such that $0 < \diam A \leq \diam B < \infty$,
  then 
  \begin{equation}
    \frac{1}{2 \eta\left( \frac{\diam B}{\diam A} \right)} \leq \frac{\diam f(A)}{\diam f(B)} \leq 
    \eta\left( \frac{2 \diam A}{\diam B} \right).
    \label{eqn:diamLemmaHein}
  \end{equation}
  \label{lem:diamLemmaHein}
\end{lemma}

\begin{lemma}
  If $X$ is a uniformly perfect metric space and $f:X \to Y$ is quasisymmetric, then $f$ is $\eta$-quasisymmetric
  with $\eta$ of the form
  \begin{equation}
    \eta(t) = C \max\left\{ t^{\beta}, t^{1/\beta} \right\},
    \label{eqn:powerQS}
  \end{equation}
  \label{lem:powerQS}
\end{lemma}
where $C\geq 1$ and $\beta \in (0,1]$ only depend on $f$ and $X$.

\begin{proof}[Proof of Lemma \ref{le:alpharegularQSinv}]
    Let $Q_{n,j},r_{n,j},d,C_1,C_2,C_3$ be as in 
    Definition \ref{def:cubes} such that $E \subset X$ is $(\alpha_n)$-regular set.
  Let $f:X \to Y$ be $\eta$-quasisymmetric, where we can assume that $\eta(t)=C\max\left\{ t^{\beta}, t^{1/\beta} \right\}$
  by Lemma \ref{lem:powerQS}.
  Obviously, we choose sets $f(Q_{n,j})$ to be the Borel sets of Definition \ref{def:cubes}, which directly gives 
  the property \ref{list:EisCubesminuscubes}. 
  We now have to show that there exist constants such
  that the properties \ref{list:cubespartition} -- \ref{list:comparableradius} are satisfied for sets $f(Q_{n,j})$.

  Properties \ref{list:cubespartition} and \ref{list:cubesdisjointandchild} follow since $f$ is a homeomorphism. 
  For property \ref{list:cubesinballs} we notice that 
  \begin{equation*}
    \frac{d(x_{n,j},z)}{d(x_{n,j},y)} \leq \frac{C_1 r_{n,j}}{ r_{n,j}} \leq C_1, \text{ for all } 
    z \in Q_{n,j}, \; y \notin Q_{n,j},
  \end{equation*}
  which implies
  \begin{equation*}
    \sup_{z \in Q_{n,j}} d(f(x_{n,j}),f(z)) \leq \eta(C_1) \inf_{y \notin Q_{n,j}} d(f(x_{n,j}),f(y)),
  \end{equation*}
  and thus 
  \begin{equation*}
    B(f(x_{n,j}),R_{n,j}) \subset f(Q_{n,j}) \subset \bar{B}(f(x_{n,j}),\eta(C_1) R_{n,j}),  
  \end{equation*}
  where $R_{n,j}=\inf_{y \notin Q_{n,j}} d(f(x_{n,j}),f(y))$. Thus we can choose $f(x_{n,j})$,$R_{n,j}$
  , $2\eta(C_1)$ as $x_{n,j}, r_{n,j},  C_{1}$ in Definition \ref{def:cubes} for the set $f(E)$.
  These notions will be used throughout the rest of the proof.

  Rephrasing Lemma \ref{lem:diamLemmaHein} with this information for 
  $x_{n,i} \in Q_{n+1,j}$ and
  $R_{n,j}$ as above gives
  \begin{equation}
    \frac{1}{4 D' \eta\left( \frac{2 D r_{n,i}}{r_{n+1,j}} \right)} \leq \frac{R_{n+1,j}}{R_{n,i}} 
    \leq 2 D' \eta\left( \frac{4 D r_{n+1,j}}{r_{n,i}} \right),
    \label{eqn:reprhasediamLemmaHein}
  \end{equation}
  where $D$ and $D'$ are uniform perfectness constants of $X$ and $Y$ respectively. Remember that $D'$ only depends 
  on $D$ and the controlling homeomorphism $\eta$.
  Now the property \ref{list:centercubealpha} for the original radii $r_{n+1,j}$ and $r_{n,i}$ together with 
  Lemma \ref{lem:powerQS} gives
  \begin{equation*}
    \begin{array}{ll}
      \frac{1}{4 D' C \max\left\{ (2 D C_2)^{\beta}, (2 D C_2)^{1/\beta}\right\} }\alpha_n^{\frac{1}{d \beta}} &\\     
     \leq \frac{1}{4 D' \eta\left( 2 D C_2 \alpha_n^{-1/d} \right)} \leq \frac{R_{n+1,j}}{R_{n,i}} 
    \leq 2 D' \eta\left( 4 D C_2 \alpha_n^d \right) \\
    \leq 2 D' C \max\left\{ (4 D C_2)^\beta, (4 D C_2)^{1/\beta} \right\} \alpha_{n}^{d \beta},
  \end{array}
  \end{equation*}
  for $f(x_{n,i}) \in f(Q_{n+1,j})$. Since $f$ is a homeomorphism, $f(Q_{n,i}) \setminus f(Q_{n+1,j}) \neq \emptyset$ if 
  and only if $Q_{n,i} \setminus Q_{n+1,j} \neq \emptyset$.
  Thus, the property \ref{list:centercubealpha} is settled with $d \beta$ and $4 D' C (4 D C_2)^{1/\beta}$ as 
  constants $d$ and $C_2$.

For \ref{list:comparableradius} we note that it is equivalent with the condition
\begin{equation}
  \forall S >1, \exists C_{}(S) \text{ such, that } 
  Q_{n,i} \subset B(x_{n,j}, S r_{n,j}) \Rightarrow r_{n,i} \geq \frac{1}{C_{}(S)} r_{n,j}.
  \label{eqn:equivCond5}
\end{equation} Indeed, the implication $\ref{list:comparableradius} \Rightarrow \eqref{eqn:equivCond5}$ is trivial.
For $\eqref{eqn:equivCond5} \Rightarrow \ref{list:comparableradius}$ assume that \eqref{eqn:equivCond5} holds
and $Q_{n,i} \cap B(x_{n,j}, T r_{n,j}) \neq \emptyset$. If $r_{n,i} \leq r_{n,j}$, we have
$Q_{n,i} \subset B(x_{n,j},3 T r_{n,j})$, which now implies $r_{n,i} \geq \frac{1}{C(3 T)} r_{n,j}$.
If $r_{n,i}>r_{n,j}$, we reverse the roles of $x_{n,j}$ and $x_{n,i}$ and get $r_{n,j} \geq \frac{1}{C(3 T)} r_{n,i}$.

%

To prove that \ref{list:comparableradius} is quasisymmetrically invariant, we show that for any choice of $S > 1$, we can 
find constant $C(S)$ such that \eqref{eqn:equivCond5} holds on the image side. Let us denote by $\eta'$ the controlling homeomorphism of $f^{-1}$ and by $D'$ the uniform perfectness constant of the space $Y$. 
Let $f(Q_{n,i}) \subset B(f(x_{n,j}), S R_{n,j})$
and note that Lemma \ref{lem:diamLemmaHein} for $f^{-1}$
implies
\begin{equation}
  \frac{1}{2 \eta'\left(\frac{ \diam B(f(x_{n,j}),S  R_{n,j}) }
  {\diam B(f(x_{n,j}),R_{n,j})}\right)} \leq
  \frac{\diam f^{-1}(B(f(x_{n,j}),R_{n,j}))}{\diam f^{-1}(B(f(x_{n,j}),S R_{n,j}))}.
  \label{eqn:useOfdiamLemma}
\end{equation}
The uniform perfectness of $Y$ implies
\begin{equation}
  \frac{\diam B(f(x_{n,j}),S R_{n,j})}{\diam B\left( f(x_{n,j}),R_{n,j} \right)} \leq 
  2 S D',
  \label{eqn:impliOfUnifPerfY}
\end{equation}
and 
$  B(f(x_{n,j}),R_{n,j}) \subset f(Q_{n,j}) $
gives
\begin{equation}
  \diam f^{-1}(B(f(x_{n,j}),R_{n,j})) \leq \diam f^{-1}(f(Q_{n,j})) \leq 2 C_1 r_{n,j}.
  \label{eqn:almostTrivi}
\end{equation}
Now, together \eqref{eqn:useOfdiamLemma}, \eqref{eqn:impliOfUnifPerfY} and \eqref{eqn:almostTrivi} imply
\begin{equation}
  \diam f^{-1}(B(f(x_{n,j}),S  R_{n,j})) \leq 2 \eta'\left( 2 S D' \right) 2 C_1 r_{n,j}.
  \label{eqn:propertyVstart}
\end{equation}
Thus, let us choose
$T:=4 C_1 \eta'\left( 2 S D' \right) $ and 
note that with this choice
$Q_{n,i} \subset B(x_{n,j},T  r_{n,j})$. 
By the property \ref{list:comparableradius} for the original set $E$ (and thus by \eqref{eqn:equivCond5}) we have for $y \notin Q_{n,i}$
\begin{equation*}
  \frac{d(x_{n,i},x_{n,j})}{d(x_{n,i},y)} \leq \frac{T  r_{n,j}}{ r_{n,i}} \leq C_1 T C_{}(T),
\end{equation*}
and this again implies
\begin{equation}
  d(f(x_{n,j}),f(x_{n,i})) \leq \eta\left(C_1 T C_{}(T)\right) R_{n,i}.
  \label{eqn:neighborcubescompar1}
\end{equation}
Also for $x_{n,j} \neq x_{n,i} \in B(x_{n,j},T C_1 r_{n,j})$ and $z \in Q_{n,j}$ we have
\begin{equation*}
  \frac{d(x_{n,j},z)}{d(x_{n,j},x_{n,i})} \leq \frac{C_1 r_{n,j}}{r_{n,j}} = C_1,
\end{equation*}
which implies
\begin{equation}
  \sup_{z \in Q_{n,j}}d(f(x_{n,j}),f(z)) \leq \eta\left( C_1 \right) d(f(x_{n,j}),f(x_{n,i})).
  \label{eqn:neighborcubescompar2}
\end{equation}

We still note that by the uniform perfectness 
for any $n$ and $j$ there exist points $z \in Q_{n,j}$ and $y \notin Q_{n,j}$  such
that $d(x_{n,j},y) \leq 2 D d(x_{n,j},z)$.
This gives
  \begin{equation}
    R_{n,j} \leq \eta(2 D) \sup_{z\in Q_{n,j}} d(f(x_{n,j}),f(z))
  \label{eqn:ballandcomplementradius}
\end{equation}
for all $n$ and $j$.

Together \eqref{eqn:ballandcomplementradius}, \eqref{eqn:neighborcubescompar1} and \eqref{eqn:neighborcubescompar2} imply
\begin{equation}
  R_{n,j} \leq \eta(2 D) \eta\left( C_1 \right) \eta\left(C_1 T C_{}(T)\right) R_{n,i}
  \label{eqn:cubescompar3}
\end{equation}
for
$f(Q_{n,i}) \subset B(f(x_{n,j}), S  R_{n,j})$,
which proves the claim.
\end{proof}

Now we note that a complete, doubling and uniformly perfect metric space is quasisymmetrically homeomorphic with 
an Ahlfors regular space. This is because a complete and doubling space carries
a doubling measure (\cite{LuukkainenSaksman1998}), and 
a uniformly perfect metric space $X$ that carries a doubling measure
is quasisymmetrically homeomorphic with an Ahlfors regular space (\cite[Corollary 14.15]{Heinonen}).
Thus, we only need to prove
Theorem \ref{thm:maintheorem} in an Ahlfors regular space, and the rest follows from quasisymmetric invariance.

\section{Additional notation}

Since we are now in a position where we only need to prove Theorem \ref{thm:maintheorem} in an Ahlfors regular space, 
let us fix one. 
  From now on, assume that $(X,d,\HH)$ is a fixed Ahlfors $q$-regular metric space, with constant $C$ 
  (and with uniform perfectness constant $D$). 
We proceed in proving \ref{list:Enotfat} and \ref{list:Enotthin} by constructing doubling measures  $\nu$ and $\mu$ 
such that 
 $\nu(E)=0$ if $(\alpha_n) \notin \ell^0$ and $\mu(E)>0$ if $(\alpha_n) \in \ell^{\infty}$ .
  
  We shall construct the desired measures in a style that resembles the Riesz product 
  (see, for example, \cite[page 182]{DavidSemmes1997}), but respects in a natural manner the geometry of our
  $(\alpha_n)$-regular set.
In addition to Definition \ref{def:cubes},
we shall be using the following notation (see Figure \ref{pic:QandI}): 
\begin{equation*}
  \I_{n,j}:=\bigcup \left\{  Q_{n+1,i} : Q_{n+1,i} \subset B(x_{n,j}, \frac{1}{2} r_{n,j}) \right\},
\end{equation*}

\begin{equation*}
  \I_{n,j}^c:= \bigcup \left\{ Q_{n+1,i} \subset Q_{n,j}: Q_{n+1,i} \not\subset \I_{n,j} \right\},
\end{equation*}

\begin{equation*}
  A_{n,j} :=   \frac{\HH(\I_{n,j})}{\int_{\I_{n,j}} d(x_{n,j},y)^{\rho} d\HH(y)},
\end{equation*}
where $ \rho > -q$ will be fixed later, depending on the sequence $(\alpha_n)$. 
The constant $A_{n,j}$ is bounded in the following sense:
There exists a constant $0<C_4<\infty$ such that
\begin{eqnarray}
  \frac{1}{C_4} r_{n,j}^{-\rho } \leq A_{n,j} \leq C_4 r_{n,j}^{-\rho }
  \label{eqn:vakioidenAkertaluokka}
\end{eqnarray}
for any $n$ and $j\in N_n$, where $\alpha_n$ is sufficiently small so that $\I_{n,j} \neq  \emptyset$.
This can be easily verified with the help of Lemma \ref{le:biggerweightsconstant} and similar computations to those in 
\eqref{eqn:waytocomputemeasureofcenter1} and \eqref{eqn:waytocomputemeasureofcenter}.

Next, let us define
    \begin{equation*}
      y_n(x) := \left\{ \begin{array}{ll} 1 &, \text{ if } x \in Q_{n+1,i} \subset \I_{n,j}^c,  \\
	A_{n,j} d(x_{n,j},x)^{\rho} &, \text{ if } x \in Q_{n+1,i} \subset \I_{n,j}. 
  \end{array}
    \right.
\end{equation*}
This should be thought of as an analog to a Jacobian of a radial stretch in each cube of level $n$.
With this as a weight we define a measure $\theta_n$ by
\begin{equation*}
  d\theta_n(x) := y_n(x) d\HH(x).
  \label{eqn:thetaDefinition}
\end{equation*}
We will also average the weight $y_n$ over each cube of the subsequent level,
by setting
\begin{equation*}
  t_n(x) := \left\{ \begin{array}{ll} 1 &, \text{ if } x \in Q_{n+1,i} \subset \I_{n,j}^c,\\
   \frac{ A_{n,j} }{\HH(Q_{n+1,i})} \int_{Q_{n+1,i}} d(x_{n,j},y)^{\rho} d\HH(y)&, \text{ if } x \in Q_{n+1,i} \subset \I_{n,j}. 
  \end{array}
    \right.
\end{equation*}
Note first that $t_n$ is constant in each $Q_{n+1,i}$ and 
that 
\begin{equation} 
  \theta_n(Q_{n,i}) = \HH(Q_{n,i}), \text{ for every }   i\in N_n 
  \label{eqn:thetaEqualsHausd}
\end{equation} 
and  
  \begin{equation}
    t_{n}(x) \HH(Q_{n+1,j}) = \theta_{n}(Q_{n+1,j}), \text{ when } x \in Q_{n+1,j}.
    \label{eqn:connectionThetaMuHausd}
  \end{equation}
  
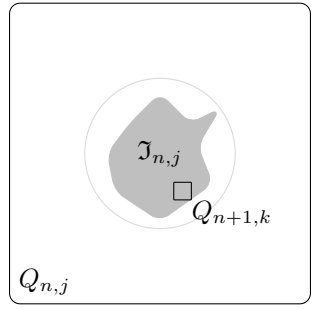
\begin{figure}[t]
\begin{tikzpicture}[rounded corners]
  \draw (-2,-2) rectangle (2,2);
  \node at (-2,-2) [above right] {$Q_{n,j}$};

\filldraw[color=gray!50] (.4,.4) -- (0,.8) -- (-.4,.4) -- (-.7,0) -- (-.6,-.5) -- (0,-.9) -- (.7,-.4) -- (.5,.1) -- (.8,.6) -- cycle;
  \draw[color=gray!30] (0,0) circle (1);
  \node at (0,0) {$\I_{n,j}$};
\node [rectangle, draw, minimum size=1pt,rounded corners=.3pt] at (.3,-.5) {};
\node [below,right] at (.3,-.8) {$Q_{n+1,k}$};

\end{tikzpicture}
  \caption{A cube $Q_{n,j}$ and set $\I_{n,j}$}
  \label{pic:QandI}
\end{figure}
In Lemma \ref{le:biggerweightsconstant} we will fix level $n_0$, which will be used as a starting point to 
mass distribution.
This is to assure that the size of cubes decreases sufficiently fast to guarantee the conclusions 
of Lemma \ref{le:biggerweightsconstant}. 
Finally,
for all $n\geq n_0 $ set 
\begin{equation}
  K_n(x) := \prod_{i=n_0}^n t_i(x),
\end{equation} 
  and with this as a weight we define a measure $\nu_n$ by
  \begin{equation} 
    d\nu_n(x) := K_n(x) d\HH(x).
    \label{eqn:nuDefinition}
  \end{equation}
  The connection between the measures $\theta_n$ and $\nu_n$ is clear from \eqref{eqn:connectionThetaMuHausd}:
  \begin{equation}
    \nu_n(Q_{n+1,j}) = K_{n-1}(x) \theta_n(Q_{n+1,j}), 
    \label{eqn:connectionThetaNu}
  \end{equation}
 for every  $j \in N_{n+1}$ and $x \in Q_{n,i} \supset Q_{n+1,j}$.

It should be understood that the purpose of the weight $K_n(x)$ is just 
to ``stretch'' the measure near the center of each cube, 
but in such a way that we can be sure to end up with a doubling
measure. 
We also have to be careful not to let the weight $K_n$ change too radically from one cube to another or on too many scales.
Otherwise we would blow up the chances to achieve a doubling measure.  
This is why we alter the original measure only in the sets $\I_{n,j}$. 
In the sets $\I_{n,j}^c$, which are close to boundaries of the cubes, the weight $t_n$ is constant
(see Figure \ref{pic:QandI}).

We still use one more notation. 
Define 
\begin{align}
  \label{eqn:INdefin}
   &IN(B(x,r),n):= \left\{  Q_{n,j}: Q_{n,j} \subset B(x,r) \right\}  \text{ and } \\
   \label{eqn:COVdefin}
   &COV(B(x,r),n):= \left\{  Q_{n,j}: Q_{n,j} \cap B(x,r) \neq \emptyset \right\}.
\end{align}
These will be used to approximate the balls $B(x,r)$ by finite unions of cubes from inside and outside, respectively.

  \section{Preliminary Lemmas}
  The claims \ref{list:Enotfat} and \ref{list:Enotthin} can both be proved in a similar manner,
  with obvious changes in certain inequalities. The idea is to use the sequence 
  $\nu_n$, defined by \eqref{eqn:nuDefinition}, and show that the weak* -limit of this 
  sequence will have the desired properties.

  The first two lemmas are technical ones. The first one simply
  fixes the level $n_0$, from where we start to redistribute the mass.

    \begin{lemma}
      There exists $n_0$, such that for all $n>n_0$ we have the following:
      \begin{enumerate}
        \item If $n$ is the first index for which  $IN(B(x,r),n) \neq \emptyset$, then
          for all $Q_{m,i} \cap B(x,2 r) \neq \emptyset, \;m \leq n-2$, we have $4 r \leq \frac{1}{2} r_{m,i}$.
        \item $B(x_{n,j},\frac{1}{16} r_{n,j}) \subset \I_{n,j}$.
        \item If $x\in Q_{n,j}$ and $r \leq \frac{1}{32 C_3(2 C_1)}r_{n,j}$, then $B(x,2 r) \cap \I_{n,i} = \emptyset$,
          for all $i \neq j$.
      \end{enumerate}
      \label{le:biggerweightsconstant}
    \end{lemma}
  \begin{proof}
    Let $x \in Q_{n-1,j}$, and since $Q_{n-1,j} \not\subset B(x,r)$, we have $r \leq 2C_{1} r_{n-1,j}$. 
    This implies $B(x,2 r) \subset B(x_{n-1,j}, 8 C_1 r_{n-1,j})$, which again implies that 
    $\frac{1}{C_3(8 C_1)} r_{n-1,k} \leq r_{n-1,j} \leq C_3(8 C_1) r_{n-1,k}$
    for all $Q_{n-1,k} \cap B(x,2 r) \neq \emptyset$. Thus 
    \begin{equation}
  r \leq 2C_{1} r_{n-1,j} \leq 2C_{1} C_3(8 C_1) r_{n-1,k}, \text{ for all } Q_{n-1,k} \cap B(x,2 r) \neq \emptyset.
  \label{eqn:RsmallerthanLastLEmma}
    \end{equation}
    Suppose $Q_{n-2,i} \cap B(x,2 r) \neq \emptyset$. Since for all 
    $Q_{n-2,i} \subset Q_{m,j}$ we have $\frac{1}{D C_1}r_{n-2,i} \leq r_{m,j}$, it is enough to show that
    $4 r \leq \frac{1}{2} \frac{1}{D C_1} r_{n-2,i}$ for any 
    $Q_{n-2,i} \cap B(x, 2 r) \neq \emptyset$.
    We now claim that $8 C_1 C_3(8 C_1) r_{n-1,k} \leq \frac{1}{2} \frac{1}{D C_1} r_{n-2,i}$ for all
    $Q_{n-1,k} \subset Q_{n-2,i}$, which together with 
    \eqref{eqn:RsmallerthanLastLEmma} would give the claim. Assume on the contrary,
    that there exists $Q_{n-1,k} \subset Q_{n-2,j}$ such that 
    $8 C_1 C_3(8 C_1) r_{n-1,k} > \frac{1}{2}\frac{1}{D C_1} r_{n-2,j}$. 
    This implies for $P:=32 D C_1^3 C_3(8 C_1)$ that $P r_{n-1,k} > 2 C_1 r_{n-2,j}$,
    and thus $Q_{n-2,j} \subset B(x_{n-1,k},P C_1 r_{n-1,k})$. In particular, \ref{list:comparableradius} 
    then implies for the center cube $Q_{n-1,k_j} \ni x_{n-2,j}$ that 
    $r_{n-1,k_j} \geq \frac{1}{C_3(P)} r_{n-1,k} > \frac{1}{16 C_1^2 C_3(P) C_3(8 C_1)}r_{n-2,j}$, 
    but \ref{list:centercubealpha} on the other hand states that $ r_{n-1,k_j}  \leq C_2 \alpha_n^d r_{n-2,j} $.
    This would be a contradiction, when $C_2 \alpha_n^d < \frac{1}{16 C_1^2 C_3(P) C_3(8 C_1)}$.
    The first condition is thus guaranteed.

     When $Q_{n+1,i} \cap X \setminus B(x_{n,j},\frac{1}{2} r_{n,j}) \neq \emptyset$ and 
     $Q_{n+1,i} \cap B(x_{n,j},\frac{1}{16} r_{n,j}) \neq \emptyset$,
     it follows that $2 C_1 r_{n+1,i} \geq \frac{7}{16} r_{n,j}.$ On the other hand, for 
     $Q_{n+1,k} \ni x_{n,j}$ we have $r_{n+1,k} \leq C_2 \alpha_n^{d} r_{n,j}$ and
     $r_{n+1,i} \leq C_3(3 C_1) r_{n+1,k}$ by \ref{list:centercubealpha} and \ref{list:comparableradius}.
     Together these imply that $\frac{7}{16} \leq C_3(3 C_1) C_2 \alpha_n^d$. 
     Recall that $\alpha_n \to 0$ as $n \to \infty$; so this quarantees the second claim.

     The third claim is true without any bound on the indices $n$. With assumptions $x \in Q_{n,j}$ and 
     $r \leq \frac{1}{32 C_3(2 C_1)}r_{n,j}$  we have that $B(x, 2 r) \subset B(x_{n,j}, 2 C_1 r_{n,j})$,
     and thus by \ref{list:comparableradius}, $r_{n,i} \geq \frac{1}{C_3(2 C_1)} r_{n,j} \geq 32 r$ for any 
     $Q_{n,i} \cap B(x,2 r) \neq \emptyset$.
     This proves the third claim.
   \end{proof}
  
  Since we do mass distribution within the cubes, the measure of a fixed cube
  does not change after finitely many steps. To prove that the resulting measures are doubling, 
  we need to be able to approximate the measures of balls 
  with the measures of cubes. For this purpose we have the next lemma. 
  \begin{lemma}
    There exist $c_9>0$ and $C_9<\infty$ such that if $IN(B(x,r),n) \neq \emptyset$, then
    $COV(B(x,2 r),n) \subset B(x,C_9 r)$ and 
    $B(x_{n,j}, c_9 r) \subset IN(B(x,r),n)$ for some $x_{n,j}$.
    \label{le:measofballbycubes} 
  \end{lemma}
  \begin{proof}
    Let $Q_{n,j} \cap B(x, 2 r) \neq \emptyset$ and $Q_{n,i} \subset B(x,r)$. Combining this with 
    \ref{list:comparableradius} and \ref{list:cubesinballs} yields $ \frac{1}{D} r_{n,j} \leq 2 C_3(4 C_1) r$.
    Thus, $COV(B(x,2 r),n) \subset B(x,(2 + 4 D C_3(4 C_1) ) r)$. This proves the first claim.

    For the other inclusion, the argument is the same as in the proof of the second claim of 
    Lemma \ref{le:biggerweightsconstant}: 
    Let $Q_{n,i} \subset IN(B(x,r),n)$. If $B(x,\frac{1}{2} r) \subset IN(B(x,r),n)$, 
    the claim is settled; so assume this does not hold. Thus, there exists a cube $Q_{n,j}$ such that
    $Q_{n,j} \cap B(x,\frac{1}{2} r) \neq \emptyset$ 
    and $Q_{n,j} \cap X\setminus B(x,r) \neq \emptyset$. 
    Now $2 C_1 r_{n,j} \geq \diam(Q_{n,j}) \geq \frac{1}{2} r$. This means that 
    $Q_{n,i} \subset B(x,r) \subset B(x_{n,j},8 C_1 r_{n,j})$, and thus by 
    \ref{list:comparableradius}: $r_{n,i} \geq \frac{1}{C_3(8 C_1)} r_{n,j} > \frac{1}{4 C_1 C_3(8 C_1)} r$. 
    This proves the second claim.
  \end{proof}

  In particular, the above lemma gives the following corollary.
  \begin{corollary}
    
  Given a doubling measure $\mu$, there exists a constant $C_{\mu} \geq 1$ 
    depending only on the doubling constant of the measure $\mu$ such that
    \begin{equation}
      \frac{1}{C_{\mu}} \mu( COV(B(x,2 r),n)) \leq \mu(B(x,r)) \leq C_{\mu} \mu(IN(B(x,r),n)),
      \label{eqn:measofballbycubes}
    \end{equation}
    whenever $IN(B(x,r),n) \neq \emptyset$.
    \label{cor:measofballbyINOUT}
  \end{corollary}

  The next lemma is the key ingredient in the proof. It essentially says that in an Ahlfors q-regular space $(X,d,\HH)$,
  a measure given by $d(x,x_0)^{\rho} d\HH(x)$ is doubling when $-q<\rho<\infty$.
  Recall that in Euclidean spaces this is well known since the radial stretch function is quasisymmetric.
  \begin{lemma}
    Let $\rho>-q$ and $n \geq n_0$. Then the measures $\theta_n$ 
    are doubling with 
    constant $C_5$, which is independent of $n$.
    \label{le:weightdoubling}
  \end{lemma}
  To start the proof, we check the doubling condition for the measure $\theta_n$ and balls centered at $x = x_{n,j}$,
  $j\in N_n$.

   \begin{lemma}
     There exists a constant $C_6 \geq 1$, such that 
     $\theta_n(B(x_{n,j},2 r)) \leq C_6 \theta_n(B(x_{n,j},r))$ for all $n\geq n_0, \; j\in N_n$ and $r>0$.
     \label{le:centeredballsdoubling}
   \end{lemma}
     We present the computations only in the case $-q < \rho \leq 0$. 
     The case $0 \leq \rho$ is left to the reader and can be proved with the same arguments,
     only by changing the estimate for distance within each annulus from inner radius to outer radius and vice versa.

   \begin{proof}
       Recall that the constant $C$ denotes the constant in the definition of the Ahlfors regular measure.
     Let $n \geq n_0 ,j\in N_n$  and $r>0$ be given. Let us first assume that $r\leq \frac{1}{32} r_{n,j}$ and 
     note that since $n \geq n_0$ in this case, $B(x_{n,j},2 r) \subset \I_{n,j}$ by Lemma \ref{le:biggerweightsconstant},
     and thus
     \begin{equation}
       \label{eqn:waytocomputemeasureofcenter1}
       \begin{array}{l}
       \theta_n(B(x_{n,j},2 r)) = A_{n,j} \int_{B(x_{n,j},2 r)} d(x_{n,j},x)^{\rho} d\HH \\
       =  A_{n,j} \sum_{i=0}^{\infty} \int_{B(x_{n,j},\frac{2 r}{p^i}) \setminus B(x_{n,j},\frac{2 r}{p^{i+1}})} d(x_{n,j},x)^{\rho} d\HH  \\
       \leq A_{n,j} \sum_{i=0}^{\infty} \int_{B(x_{n,j},\frac{2 r}{p^i}) \setminus B(x_{n,j},\frac{2 r}{p^{i+1}})} \left( \frac{2 r}{p^{i+1}} \right)^{\rho} d\HH \\
       \leq A_{n,j} \sum_{i=0}^{\infty} \left( \frac{2 r}{p^{i+1}} \right)^{\rho} \left( C\left( \frac{2 r}{p^i} \right)^q - \frac{1}{C} \left( \frac{2 r}{p^{i+1}} \right)^q \right) \\
       =  A_{n,j} p^{-\rho} (2 r)^{\rho+ q} \left( C - \frac{1}{C p^q} \right) \sum_{i=0}^{\infty} p^{-i (\rho + q)},
     \end{array}
     \end{equation}
     and on the other hand, 
     \begin{equation}
       \label{eqn:waytocomputemeasureofcenter}
       \begin{array}{l}
       \theta_n(B(x_{n,j}, r)) = A_{n,j} \int_{B(x_{n,j}, r)} d(x_{n,j},x)^{\rho} d\HH \\ 
       = A_{n,j} \sum_{i=0}^{\infty} \int_{B(x_{n,j},\frac{r}{p^i}) \setminus B(x_{n,j},\frac{ r}{p^{i+1}})} d(x_{n,j},x)^{\rho} d\HH  \\
       \geq A_{n,j} \sum_{i=0}^{\infty} \int_{B(x_{n,j},\frac{r}{p^i}) \setminus B(x_{n,j},\frac{ r}{p^{i+1}})} \left( \frac{r}{p^i} \right)^{\rho} d\HH  \\
       \geq A_{n,j} \sum_{i=0}^{\infty} \left( \frac{r}{p^{i}} \right)^{\rho} \left( \frac{1}{C}\left( \frac{r}{p^i} \right)^q - C \left( \frac{r}{p^{i+1}} \right)^q \right) \\
       =  A_{n,j} r^{\rho+ q} \left( \frac{1}{C} - \frac{C}{p^q} \right) \sum_{i=0}^{\infty} p^{-i (\rho + q)},
     \end{array}
     \end{equation}
     where $p\in\N$ is chosen to be such that $\frac{1}{C} - \frac{C}{p^q}>0$. 

     If $r\geq C_1 r_{n,i}$, it follows that we are really dealing with the original Ahlfors regular measure $\HH$,
     and the result follows
     from \eqref{eqn:measofballbycubes} and \eqref{eqn:thetaEqualsHausd}.

     Now we are left with the case $ \frac{1}{32}  r_{n,j} < r \leq C_1 r_{n,j}  $.
     In this case, we simply compare the measure of a ball with the measure of whole $Q_{n,j}$.
     Indeed, we compute as in \eqref{eqn:waytocomputemeasureofcenter}:
     \begin{eqnarray*}
       \theta_n(B(x_{n,j}, r)) \geq \theta_n(B(x_{n,j}, \frac{1}{32} r_{n,j})) 
       \geq A_{n,j}C_p \left(r_{n,j} \right)^{\rho + q}
       \geq C_p  \frac{1}{C_4} r_{n,j}^q, 
     \end{eqnarray*}
     where $C_p =  
     \left( \frac{1}{32} \right)^{\rho + q}\left( \frac{1}{C} - \frac{C}{p^q} \right) \sum_{i=0}^{\infty} p^{-i (\rho + q)}$, while 
     \begin{eqnarray*}
       \theta_n(B(x_{n,j},2 r)) \leq \theta_n(COV(B(x_{n,j},2 r),n)) = \HH(COV(B(x_{n,j},2 r),n)) \\
       \leq \HH(COV(B(x_{n,j},2 C_1 r_{n,j}),n)) \leq C_{\HH} C (2 C_1 r_{n,j})^q,
     \end{eqnarray*}
     by  \eqref{eqn:thetaEqualsHausd} and \eqref{eqn:measofballbycubes} and \eqref{eqn:ahlforsdefinition}.
   \end{proof}

   Now we proceed with the proof of Lemma \ref{le:weightdoubling}.

   \begin{proof}[Proof of Lemma \ref{le:weightdoubling}]
   Let $x \in Q_{n,j}$ and first assume 
   $r \leq \frac{1}{32 C_3(2 C_1)}  r_{n,j}$. By Lemma \ref{le:biggerweightsconstant} 
   this implies that $B(x,2 r) \cap \I_{n,i} = \emptyset$ for
   all $i \neq j$. If also $B(x, 2r) \cap \I_{n,j} = \emptyset$, the weight $y_n = 1$, and thus $d\theta_n=d\HH$ 
   in the whole $B(x, 2 r)$. 
   Let us now notice that 
   if $B(x,2 r) \cap \I_{n,j} \neq \emptyset \neq B(x,2 r) \cap \I_{n,j}^{c}$, it follows that
   \begin{eqnarray*}
 B(x,2 r) \subset B(x_{n,j},\frac{1}{2} r_{n,j} + 2 r) \setminus B(x_{n,j}, \frac{1}{4} r_{n,j} - 2 r)\\
  \subset B(x_{n,j}, r_{n,j}) \setminus B(x_{n,j},\frac{1}{8} r_{n,j}). 
   \end{eqnarray*}
   This again implies that
   \begin{equation*}
 \frac{1}{8} r_{n,j} \leq d(y,x_{n,j}) \leq r_{n,j}, \; \forall y \in B(x,2 r),
   \end{equation*}
   and thus by \eqref{eqn:vakioidenAkertaluokka},
   \begin{equation}
 \label{eqn:estimatesforweight1}
 \frac{1}{C_4} \leq A_{n,j} d(y,x_{n,j})^{\rho} \leq C_4 \left(\frac{1}{8} \right)^{\rho}, \;
 \forall y \in B(x, 2 r),
   \end{equation} when $-q < \rho \leq 0$ and 
   \begin{equation}
 \label{eqn:estimatesforweight2}
 \left(\frac{1}{8} \right)^{\rho}\frac{1}{C_4} \leq A_{n,j} d(y,x_{n,j})^{\rho} \leq C_4 ,\;
 \forall y \in B(x, 2 r),
   \end{equation}
   when $0\leq \rho $.
   Thus, the weight $y_{n}$ is bounded from below and from above in the whole $B(x,2r)$ 
   with constants that are 
   independent of $n$, and this  
   ensures the measure $\theta_n$ to be doubling in this case, with a constant independent of $n$.
   
  So, to complete the proof assuming $x \in Q_{n,j}$ and $r \leq \frac{1}{32 C_3(2 C_1 )} r_{n,j}$,
  we may now assume that $B(x,2 r) \subset \I_{n,j}$. We split the proof into three subcases:

	\begin{enumerate}
        \renewcommand{\theenumi}{Case \arabic{enumi}}
        \renewcommand{\labelenumi}{\theenumi}
     \item 
       $d(x_{n,j},x) < \frac{1}{2} r$
       \label{list:Case1}
     \item
       $\frac{1}{2} r \leq d(x_{n,j},x) \leq 4 r$
       \label{list:Case2}
     \item
       $d(x_{n,j},x) > 4 r$
       \label{list:Case3}
	\end{enumerate}
	Let us first go through \ref{list:Case1}-\ref{list:Case3} when $-q < \rho \leq 0$. 
	The obvious changes when $0<\rho$ are left to 
the reader.
For \ref{list:Case1}
we note that 
\begin{eqnarray*}
	 B(x_{n,j}, \frac{1}{4} r) \subset B(x,r), \\
	 B(x, 2 r) \subset B(x_{n,j},4 r)
       \end{eqnarray*}
       and thus we get by Lemma \ref{le:centeredballsdoubling}
       \begin{equation*}
	 \theta_n(B(x, 2 r)) \leq \theta_n(B(x_{n,j},4 r)) 
	 \leq C_6^4 \theta_n(B(x_{n,j},\frac{1}{4} r)) \leq C^2 \theta_n(B(x,r)).
       \end{equation*}

       For \ref{list:Case2} we estimate
       $B(x, 2 r) \subset B(x_{n,j},6 r)$ and compute as in \eqref{eqn:waytocomputemeasureofcenter1}:
       \begin{equation}
	 \label{eqn:case2approxabove}
	 \theta_n(B(x, 2r)) \leq \theta_n(B(x_{n,j},6 r)) = A_{n,j} \int_{B(x_{n,j}, 6 r)} d(x_{n,j},x)^{\rho} d\HH
	 \leq C_p r^{q+\rho},
       \end{equation}
       where $C_p=A_{n,j} p^{-\rho} 6^{\rho+ q} \left( C - \frac{1}{C p^q} \right) \sum_{i=0}^{\infty} p^{-i (\rho + q)}.$ 
       On the other hand, we also have  $\frac{1}{4} r \leq d(y,x_{n,j}) \leq 5r$ for all 
       $y \in B(x,\frac{1}{4} r)$, which implies
       \begin{equation}
	 \label{eqn:case2approxbelov}
	 \theta_n(B(x, r)) \geq \theta_n(B(x,\frac{1}{4} r)) = A_{n,j} \int_{B(x, \frac{1}{4} r)} d(x_{n,j},x)^{\rho} d\HH\\
	 \geq  A_{n,j} \frac{1}{C} (\frac{1}{4} r)^q (5 r)^{\rho} ,
       \end{equation}
       which together give the result.

       \ref{list:Case3} is proved by noting that the weight is essentially a constant inside $B(x,2 r)$.
       For all $j \in N_n$,
       $\forall z \in B(x,2 r), \; \forall y \in B(x,2 r)$, we see that
       \begin{equation*}
	 \begin{array}{l}
	 d(x_{n,j},z) \leq 4 d(x_{n,j}, y), 
       \end{array}
       \end{equation*}
       and thus by the $q$ -regularity of $\HH$
       \begin{equation*}
	 \begin{array}{l}
	 \theta_n(B(x,2 r)) \leq \int_{B(x,2 r)} \sup_{z \in B(x, 2 r)} y_n(z) d\HH \\
	 \leq 4^{-\rho} \int_{B(x,2 r)}\inf_{z \in B(x,2 r)} y_n(z) d\HH 
	 \leq C^2 2^q 4^{-\rho} \theta_n(B(x,r)).
       \end{array}
       \end{equation*}
       
       Now we are left with the range of radii
       $\frac{1}{32 C_3(2 C_1)} r_{n,j} < r $. If $ 2 C_1 r_{n,j} < r$, then $IN(B(x,r),n)\neq \emptyset$,
        and  by \eqref{eqn:measofballbycubes}
       and \eqref{eqn:thetaEqualsHausd} we get
       \begin{equation*}
	 \begin{array}{l}
	 \theta_n(B(x,2 r)) \leq \theta_n(COV(B(x,2 r),n)) = \HH(COV(B(x, 2 r),n)) \\
	 \leq C_{\HH}^2 \HH(IN(B(x,r),n)) = C_{\HH}^2\theta_n(IN(B(x,r),n)) \leq C_{\HH}^2 \theta_n(B(x,r)).
       \end{array}
       \end{equation*}
       For $-q < \rho \leq 0$ and
       $\frac{1}{32 C_3(2 C_1)} r_{n,j} < r <2 C_{1} r_{n,j}$ we can estimate the $\theta_n$ measure of $B(x,2 r)$
       from above in the same way
       as before. More precisely, by \eqref{eqn:measofballbycubes} and \eqref{eqn:thetaEqualsHausd}, we have
       \begin{equation}
	 \label{eqn:2Bmeasurefromabovebycover}
	 \begin{array}{l}
	   \theta_n(B(x,2 r)) \leq \theta_n(B(x,4 C_1 r_{n,j})) \leq \theta_n(COV(B(x,4 C_1 r_{n,j})),n) \\
	   = \HH((COV(B(x,4 C_1 r_{n,j})),n)) \leq C_{\HH} C (2 C_1 r_{n,j})^{q} \leq  C_{\HH} C (64 C_1 C_3(2) r)^{q}. 
       \end{array}
       \end{equation}
       To get an estimate from below, we use the rough estimate 
       $d(x_{n,j},y)^{\rho} \geq \frac{1}{2} r_{n,j}^{\rho}$ for every $y \in \I_{n,j}$ and by 
       \eqref{eqn:ahlforsdefinition} and \eqref{eqn:vakioidenAkertaluokka}: 
       \begin{equation*}
         \theta_n(B(x,r)) \geq \max\left\{A_{n,j} \left( \frac{1}{2}  r_{n,j} \right)^{\rho}, 1 \right\} \HH(B(x,r)) 
         \geq \max\left\{ \left(\frac{1}{2}  \right)^{\rho} \frac{1}{C_4},1 \right\} \frac{1}{C} r^{q} .
       \end{equation*}
       This completes the case $-q<\rho\leq0$.

       Let us finally check the case $\rho>0$ and
       $\frac{1}{32 C_3(2 C_1)} r_{n,j} < r <2 C_{1} r_{n,j}$. 
       As in \eqref{eqn:2Bmeasurefromabovebycover}, we get 
       $\theta_n(B(x,2 r)) \leq  C_{\HH} C (2 C_1 r_{n,j})^{q}$.
       From below, we estimate the measure $\theta_n(B(x,r))$ by $\theta_n(B(x,\frac{1}{32 C_3(2 C_1)} r_{n,j}))$
       and consider the same three cases as above. Like before, we can all the time assume that 
       $B(x,\frac{1}{32 C_3(2 C_1)} r_{n,j}) \subset \I_{n,j}$, otherwise the
       estimate \eqref{eqn:estimatesforweight2} yields 
       the result.

       For \ref{list:Case1} we approximate by a smaller $x_{n,j}$ centered ball and 
       compute as in \eqref{eqn:waytocomputemeasureofcenter}:
       \begin{equation}
	 \begin{array}{l}
	 \theta_n(B(x,\frac{1}{32 C_3(2 C_1)} r_{n,j})) \geq
	 \theta_n(B(x_{n,j},\frac{1}{4}\frac{1}{32 C_3(2 C_1)}  r_{n,j})) \\
	 \geq A_{n,j} C_p  \left(\frac{1}{4} \frac{1}{32 C_3(2 C_1)}  r_{n,j}\right)^{\rho + q} 
	 \geq C_p p^{-\rho} \frac{1}{C_4} \left( \frac{1}{4} \frac{1}{32 C_3(2 C_1)}  \right)^{\rho} r_{n,j}^{q},  
       \end{array}
       \end{equation}
       where $C_p=p^{-\rho} \left( \frac{1}{C} - \frac{C}{p^q} \right) \sum_{i=0}^{\infty} p^{-i (\rho + q)}$.

       For \ref{list:Case2} we compute as in \eqref{eqn:case2approxbelov}:  
       \begin{equation*}
	 \begin{array}{l}
	 \theta_n(B(x, \frac{1}{32 C_3(2 C_1)}  r_{n,j})) \\
	 \geq  A_{n,j} \frac{1}{C} (\frac{1}{4}  \frac{1}{32 C_3(2 C_1)}  r_{n,j})^q (\frac{1}{4} \frac{1}{32 C_3(2 C_1)}  r_{n,j})^{\rho} \\
	 \geq \frac{1}{C C_4} (\frac{1}{4}  \frac{1}{32 C_3(2 C_1)}  )^q (\frac{1}{4} \frac{1}{32 C_3(2 C_1)} )^{\rho}r_{n,j}^q.
       \end{array}
       \end{equation*}

       Finally for \ref{list:Case3},
       we note that $d(x_{n,j},x) > 4 r \geq \frac{1}{8}  r_{n,j}$, and thus $d(x_{n,j}, y) \geq \frac{1}{8}  r_{n,j} - \frac{1}{32 C_3(2 C_1)}  r_{n,j} \geq \frac{3}{32 C_3(2 C_1)}  r_{n,j}$, for all $y \in B(x,\frac{1}{32 C_3(2 C_1)}  r_{n,j})$. Now we can compute
       \begin{equation*}
	 \begin{array}{l}
	 \theta_n(B(x,\frac{1}{32 C_3(2 C_1)}  r_{n,j} )) \\
	 \geq A_{n,j} \int_{B(x,\frac{1}{32 C_3(2 C_1)}  r_{n,j})} d(x_{n,j},y)^{\rho} d\HH \\
	 \geq A_{n,j} \frac{1}{C} \left( \frac{1}{32 C_3(2 C_1)}  r_{n,j} \right)^q \left( \frac{3}{32 C_3(2 C_1)}  r_{n,j} \right)^{\rho} \\
	 \geq \frac{1}{C C_4} \left( \frac{1}{32 C_3(2 C_1)}  \right)^q \left( \frac{3}{32 C_3(2 C_1)} \right)^{\rho}r_{n,j}^{q}. 
       \end{array}
       \end{equation*}
       This completes the proof.
  \end{proof}

The next thing is almost like saying that the measure $t_n d\HH$ would also be doubling. 
We do not exactly prove this, but instead something a bit stronger, 
namely that the weight $t_n$ is comparable to a constant in balls small enough compared to $r_n$. 

  \begin{lemma}
    There exists a constant $C_7$ such that $\frac{1}{C_7} t_{n-1}(y) \leq t_{n-1}(x) \leq C_7 t_{n-1}(y)$
    for all $x\in Q_{n,i}$, $y \in Q_{n,j}$, whenever 
    $ B(x_{n,k},8C_1 r_{n,k}) \cap Q_{n,j} \neq \emptyset \neq B(x_{n,k},8C_1 r_{n,k}) \cap Q_{n,i}$
    for some $x_{n,k}$.
     
    \label{le:weightscomparable}
  \end{lemma}
  \begin{proof}
    Let $x\in Q_{n,i}$, $y \in Q_{n,j}$ be such that $t_{n-1}(x)=\sup_{z \in B(x_{n,k}, 8 C_1 r_{n,k})} t_{n-1}(z)$ and  $t_{n-1}(y)=\inf_{z \in B(x_{n,k}, 8 C_1 r_{n,k})} t_{n-1}(z)$. 
    We use the fact that $\theta_n$ is doubling with \eqref{eqn:connectionThetaMuHausd} and properties 
    \ref{list:comparableradius} and \ref{list:cubesinballs} and compute 
    \begin{equation*}
      \begin{array}{l}
      \frac{1}{C C_3(8 C_1)^q} r_{n,k}^q \sup_{z \in B(x_{n,k}, 8 C_1 r_{n,k})} t_{n-1}(z) \\
      \leq \HH(Q_{n,i}) \sup_{z \in B(x_{n,k}, 8 C_1 r_{n,k})} t_{n-1}(z) = t_{n-1}(x) \HH(Q_{n,i})  \\
      =\theta_{n-1}(Q_{n,i}) \leq \theta_{n-1}(COV(B(x_{n,k},8 C_1 r_{n,k}),n) \\
      \leq \theta_{n-1}(B(x_{n,j},64 C_3(8 C_1)^2  C_1 r_{n,j})) \leq C_{\theta} \theta_{n-1}(B(x_{n,j}, r_{n,j})) \\
      \leq C_{\theta} \theta_{n-1}(Q_{n,j}) = C_{\theta} t_{n-1}(y) \HH(Q_{n,j})  \\
      = C_{\theta} \inf_{z \in B(x_{n,j}, 8 C_1 r_{n,k})} t_{n-1}(z) \HH(B(x_{n,j},C_1 r_{n,j})) \\
      \leq C_{\theta} C C_1^q C_3(8 C_1)^q r_{n,k}^q \inf_{z \in B(x_{n,k}, 8 C_1 r_{n,k})} t_{n-1}(z),
    \end{array}
    \end{equation*}
    where $C_{\theta}$ depends on the doubling constant of $\theta_n$, and is independent of $n$ by 
    Lemma \ref{le:weightdoubling}. 
  \end{proof}
  
  \begin{corollary}
    There exists a constant $C_8$ such that if $n$ is the first index for which $IN(B(x,r),n) \neq \emptyset$, then 
    \begin{equation*}
      \sup_{z\in B(x, 2 r)} t_{n-2}(z) \leq C_8 \inf_{z \in B(x,r)} t_{n-2}(z).
    \end{equation*}
    \label{cor:comparableweights}
  \end{corollary}
  \begin{proof}
    Let $x \in Q_{n-1,i}$.
    Since $Q_{n-1,i} \not \subset B(x,r)$, we have $2 C_1 r_{n-1,i} > r$. 
    Thus for any $Q_{n-1,k} \cap B(x, 2 r)\neq \emptyset$ we have
    $Q_{n-1,k} \cap B(x_{n-1,i}, 8 C_1 r_{n-1,i})\neq \emptyset$
    and the result follows from Lemma \ref{le:weightscomparable}.
  \end{proof}

\section{Proof of the Theorem \ref{thm:maintheorem}}

With the help of the previous lemmata and definitions we are now able to finish the proof of Theorem \ref{thm:maintheorem}. 
Let us consider the sequence of measures $\nu_n$, defined in \eqref{eqn:nuDefinition}. 
By the Banach-Alaoglu theorem we know that there exists a subsequence $\nu_{n_k}$ converging in the 
weak* sense to a measure $\nu$ (see, for example, \cite{AmbTil}).
It is actually true (as a result of how the mass is distributed from $\nu_{n-1}$ to $\nu$) 
that the whole sequence $\nu_n$ converges to the same limit,
  but we do not need this stronger result; any weak* limit will be good for our purposes.
  Furthermore, it is well known that if measures $\nu_n$ are doubling with the same constant $C$, then also 
  the weak* limit $\nu$
  is doubling. This is easy to see: Remember that weak* convergence is equivalent to 
  $\nu(U)\leq {\lim \inf}_{n} \nu_n(U)$ for any open set $U$
  and $\nu(K) \geq {\lim \sup}_{n} \nu_n(K)$ for any compact set $K$. 
  Thus $\nu(B(x,2r)) \leq {\lim \inf}_n \nu_n(B(x,2r)) \leq C^2 {\lim \sup}_n\nu_n(\bar{B}(x,\frac{1}{2}r)) \leq C^2 \nu(\bar{B}(x,\frac{1}{2}r)) \leq C^2 \nu(B(x,r))$.
  So, to prove that the measure $\nu$ is doubling,
  we show that the measures $\nu_n$ are doubling with the same constant.

  \begin{lemma}
    There exists a constant $C_{\nu}\geq 1$ such that all the measures $\nu_n, \; n\geq n_0$ are  
    $C_{\nu}$-doubling.
    \label{le:convergmeasUnivDoubl}
  \end{lemma}
  \begin{proof}
    Let $x \in X$ and $r>0$ be given.
    Let $n$ be the first index for which $Q_{n,j} \subset B(x,r)$ for some $j \in N_n$. 
    Since $\nu_m(Q_{n,j}) = \nu_{n-1}(Q_{n,j})$ for all $m\geq n-1$ and for all $j \in N_{n-1}$, it is enough
    to show that $\nu_{n-1}(COV(B(x, 2 r),n)) \leq C \nu_{n-1}(IN(B(x,r),n))$.
    
    By Corollary \ref{cor:comparableweights} we know that at each point $y$ in $B(x, 2 r)$ the weight $t_{n-2}(y)$ is 
    essentially constant  (and thus in $COV(B(x,2 r),n)$, as well).
    For the previous weights $t_k, k<n-2$, it is easy to see that the weight has to be constant at whole $B(x, 2 r)$: 
    
    Let us first assume that $\I_{n-2,j} \cap B(x,2 r) \neq \emptyset$.
    For all $y \in B(x, 2 r)$ we thus have $d(y,x_{n-2,j}) < 4 r + \frac{1}{2}  r_{n-2,j} \leq   r_{n-2,j}$, by Lemma \ref{le:biggerweightsconstant}.
    Thus $B(x,2 r) \subset Q_{n-2,j}$ and thus $t_{n-3}(x)=C$ for all $x \in B(x, 2 r)$. 
    If on the other hand $\I_{n-2,j} \cap B(x,2 r) = \emptyset$ for all $j \in N_{n-2}$, we know that the weight $t_{n-2}(z)=1$ in whole $B(x, 2 r)$ and argue the same way for $n-3$. Thus we see that all the
    preceding weights $t_k, \; k \leq n-3$ are constants in whole $B(x, 2 r)$.  

    With this in mind we can use \eqref{eqn:connectionThetaNu} and compute
    \begin{equation*}
      \begin{array}{l}
	\nu_{n-1}(COV(B(x, 2 r),n)) \leq \prod_{i=0}^{n-2} \sup_{z \in B(x,2 r)}t_{i}(z) \theta_{n-1}(COV(B(x,2 r),n))\\
	\leq C_8 \prod_{i=0}^{n-2} \inf_{z \in B(x, r)}t_{i}(z) \theta_{n-1}(COV(B(x,2 r),n)) \\ 
	\leq C_8 C_{\theta_{n-1}}^2 \prod_{i=0}^{n-2} \inf_{z \in B(x, r)}t_{i}(z) \theta_{n-1}(IN(B(x,r),n)) \\ 
	\leq C_8 C_{\theta_{n-1}}^2  \nu_{n-1}(IN(B(x,r),n)),
      \end{array}
    \end{equation*}
    where the constants $C_8$ and $C_{\theta_{n-1}}$ are from Corollaries \ref{cor:comparableweights} and 
     \ref{cor:measofballbyINOUT}. Remember that $C_{\theta_{n-1}}$ only depends on the doubling 
     constant of the measures $\theta_{n-1}$ and
    by Lemma \ref{le:weightdoubling} these are uniformly bounded by $C_5$.
  \end{proof}


Let $\nu$ be the weak* limit of measures $\nu_n$, with $0<\rho$ such 
that $\sum_{n=0}^{\infty} \alpha_n^{(q + \rho)d} < \infty$.
We can now finish the proof of the claim \ref{list:Enotthin}.   
Note that the boundaries $\partial Q_{n,j}$
are upper porous: If $x \in \partial Q_{n,j}$, take a sequence 
$(2 C_1 r_{m,i})_{m=n}^{\infty}$ for $Q_{m,i}$ such that $x \in \partial Q_{m,i}$. For this sequence 
$B(x_{m,i},  r_{m,i}) \subset B(x,2 C_1 r_{m,i}) \setminus \partial Q_{n,j}$.
It is well known that upper porosity implies that the set is thin. 
Thus, $\partial Q_{n,j}$ is of measure zero for doubling measures for all $n \in \N, \; j\in N_n$. In particular,
this guarantees that $\lim_j \nu_j(E_n) = \nu(E_n)$ for all $n\in\N$.
Now for each cube $Q_{n,j}$, the measure of  the 
removed center cube 
$Q_{n+1,i_j}$ can be approximated  by 
\begin{equation}
  \begin{array}{l}
    \nu(Q_{n+1,i_j}) = \nu_n(Q_{n+1,i_j}) = K_{n-1}(x_{n+1,i_j}) \theta_n(Q_{n+1,i_j}) \\
  \leq  K_{n-1}(x_{n+1,i_j})\int_{B(x_{n+1,i_j}, C_1 r_{n+1,i_j})} A_{n,j} d(x,x_{n,j})^{\rho} d\HH(x)  \\
  \leq K_{n-1}(x_{n+1,i_j})\int_{B(x_{n+1,i_j}, C_1 r_{n+1,i_j})} A_{n,j} (2 C_1 r_{n+1,i_j})^{\rho} d\HH(x)   \\
  \leq K_{n-1}(x_{n+1,i_j})C C_4  2^{\rho} C_1^{\rho +q} r_{n,j}^{-\rho}
  r_{n+1,i_j}^{q + \rho}   \\
  \leq K_{n-1}(x_{n+1,i_j}) C_E \alpha_n^{(q + \rho) d} r_{n,j}^{q}\\
  \leq K_{n-1}(x_{n+1,i_j}) C_E \alpha_n^{(q + \rho) d} C \theta_n(Q_{n,j})\\
  = C_E C\alpha_n^{(q + \rho) d}  \nu_n(Q_{n,j}) = C_E C\alpha_n^{(q + \rho) d}  \nu(Q_{n,j}),
\end{array}
\label{eqn:Qn+1AprByQn}
\end{equation}
where $C_E=C C_4  2^{\rho} C_1^{\rho +q} C_2^{\rho + q} $,
and the inequalities and equalities follow from the above reasoning,
\eqref{eqn:connectionThetaNu}, \eqref{eqn:thetaEqualsHausd}, \ref{list:centercubealpha}, \eqref{eqn:ahlforsdefinition} and 
\eqref{eqn:vakioidenAkertaluokka}.
  
  With this in mind we can approximate in each $Q_{1,k}$
\begin{equation*}
  \begin{array}[c]{l}
      \nu(E_{n}  \cap E_{n-1} \cap Q_{1,k}) = \nu(\bigcup_{Q_{n,j} \subset E_{n-1} \cap Q_{1,k}}(Q_{n,j} \setminus Q_{n+1,i_j}) \\
    = \sum_{Q_{n,j} \subset E_{n-1} \cap Q_{1,k}} \left( \nu(Q_{n,j}) - \nu(Q_{n+1,i_j})\right) \\
    \geq \left( 1 - C_E C \alpha_n^{(q + \rho) d} \right) \nu(\bigcup_{Q_{n,j} \subset E_{n-1}\cap Q_{1,k}} Q_{n,j}) \\
    =  \left( 1 - C_E C  \alpha_n^{(q + \rho) d} \right) \nu(E_{n-1} \cap Q_{1,k}).
  \end{array}
  \label{eqn:approxofEnbyEn-1}
\end{equation*}
When we apply this for all $n \geq n_1 >n_0$, where $n_1$ is such that 
$ 1 - C_E C \alpha_n^{(q + \rho) d} > 0$, for all $n\geq n_1$ we get 
\begin{equation}
    \begin{array}[c]{l}
     \nu(E \cap Q_{1,k}) = \lim_{i \to \infty} \nu\left( \cap_{j=1}^{i} E_{j} \cap Q_{1,k} \right) \\
        \geq  \prod_{j=n_1}^{\infty}\left( 1 - C_E C \alpha_{j}^{(q + \rho) d} \right) 
        \nu(\cap_{j=1}^{n_1} E_{j} \cap Q_{1,k}) > 0, 
  \label{eqn:finalstep1}
  \end{array}
\end{equation}
because $\sum_{j=1}^{\infty} \alpha_j^{(q+\rho) d} < \infty$ .

The claim \ref{list:Enotfat} can be proved with similar calculations as above. 
If $(\alpha_n) \notin \ell^{0}$, we choose $-q < \rho < 0$ such that 
$\sum_{n=0}^{\infty} \alpha_n^{\frac{q + \rho}{d}} = \infty$, and this yields
$\nu(E)=0$ for the limiting measure.  
\qed


\section{Measures on the $(\alpha_n)$ -regular set $E$}
We keep the notation from the previous sections and consider 
the measure $\nu$ (which was constructed in the proof of Theorem \ref{thm:maintheorem}) as a measure 
on the $(\alpha_n)$ -regular set $E$. By this we mean
restricting the measure $\nu$ and the metric $d$ to the set $E$ and considering $(E,\nu\rst{E},d\rst{E})$ as 
a metric measure space of its own. We are all the time dealing with
the case $(\alpha_n) \in \ell^{\infty}$, otherwise the set would be thin (and 
thus the restricted measure trivial).

Two closely related concepts are the measure density condition (see, for example, \cite{HajKosTuo}) and the
plumpness (see \cite{HytKai} and references therein) of a set. 
If either of these two conditions were satisfied, we could easily conclude 
that $\nu$ is doubling as a measure on the set $E$. 
This is not the case, since the set $E$ is obviously not plump, and even the measure density 
property migh not be satisfied, as Example \ref{ex:nondoubling} shows.
Let us first record the following corollary of the proof of Theorem \ref{thm:maintheorem} which 
quantifies the approximative measure density. The summability condition below is simply to
guarantee that our measure is not trivial. Thus, we have a collection of measures $\nu$ that we are interested in:
one for each $\rho$ such that the sum $\sum_{i=0}^{\infty} \alpha_{j}^{\left( q + \rho \right) d }$ is finite.

\begin{corollary}
    Suppose that $X$ is a bounded, uniformly perfect metric space, $E \subset X$ is an $(\alpha_n)$-regular set
    and $\nu$ is a doubling measure defined as in the proof of 
    Theorem \ref{thm:maintheorem} with $\rho$ such that
    $\sum_{i=0}^{\infty} \alpha_{j}^{\left( q + \rho \right) d } < \infty$. Then there exists 
    a constant $c>0$ such that $\nu(Q_{n,k} \cap E) \geq c \nu(Q_{n,k})$
    for all $Q_{n,k} \subset \cap_{i=1}^{n-1} E_{i}$.
    \label{co:moreThanConstantLeft}
\end{corollary}
\begin{proof}
    Exactly as in \eqref{eqn:finalstep1}, we see that
    \begin{equation*}
        \nu(E \cap Q_{n,k}) = \lim_{i \to \infty} \nu\left( \cap_{j=1}^{i} E_{j} \cap Q_{n,k} \right) 
        \geq c_1 \nu(\cap_{j=1}^{n_1} E_{j} \cap Q_{n,k}), 
    \end{equation*}
    for any $Q_{n,j} \subset \cap_{i=1}^{n-1} E_i$
    where $c_1 = \prod_{j=n_1}^{\infty}\left( 1 - C_E C \alpha_{j}^{(q + \rho) d} \right)$.
    If $n > n_1$, we are done, since in this case $  Q_{n,k} \subset \cap_{j=1}^{n_1} E_{j}$. If $n \leq n_1$, 
    since we have only finite number of cubes $Q_{n,j}, n\leq n_1$ (we assume that our space is bounded), 
    we have $\min_{j, n\leq n_1} \left\{ \nu(E\cap Q_{n,j} )/\nu(Q_{n,j})\right\} = c_2> 0$. 
    These together give a lower bound to how much 
    we have at most removed from any cube. 
\end{proof}

With the above corollary we see that, if we assume certain weak quantitative plumpness 
from the approximating sets $\cap_{i=1}^n E_i$,
then the measures as in the previous corollary are doubling as measures on the set $E$. 
We state this sufficient requirement in the next definition. Compared to plumpness in \cite{HytKai}, we are 
basically only checking that $\cap_{i=1}^n E_i$ looks plump at the very coarse scale of radii.
\begin{definition}
    An $(\alpha_n)$-regular set $E \subset X$ is \emph{relatively plump} if there exists constant $b>0$ such that following 
    condition is satisfied:
    For any $x \in E$ and $R>0$ and for the first $n \in \N$
    for which $\exists Q_{n,j} \subset B(x,R) \cap \cap_{i=1}^{n-1} E_{i}$, there exists 
    $y \in \cap_{i=1}^{n-1} E_i$ such that $B(y, b R) \subset B(x, R) \cap \cap_{i=1}^{n-1} E_{i}$.
\end{definition}
Th key point is that $r_{n,j}$ is not necessarily comparable to $R$ in the above definition when $\alpha_n \to 0$. 
As Example \ref{ex:nondoubling} shows, if this relative plumpness is not satisfied, 
the doubling measures restricted to 
the set $E$ might not be doubling (as measures on $E$).
\begin{proposition}
    Suppose $X$ is a bounded, uniformly perfect metric space 
    and $E \subset X$ is a relatively plump $(\alpha_n)$-regular set. Then
    the measures $\nu$ defined as in the proof of 
    Theorem \ref{thm:maintheorem} with $\rho$ such that
    $\sum_{i=0}^{\infty} \alpha_{j}^{\left( q + \rho \right) d } < \infty$
    are doubling as measures on the set $E$.
        \label{le:measureDoublingOfE}
\end{proposition}

\begin{proof}
    Let $x \in E$, $R>0$ and let $n$ be the smallest index for which 
    $\exists Q_{n,j} \subset B(x,R) \cap \cap_{i=1}^{n-1} E_{i}$.
    By assumption there now exists $y \in \cap_{i=1}^{n-1} E_{i}$ such that 
    $B(y,b R)\subset B(x,R) \cap \cap_{i=1}^{n-1} E_{i}$.
    If $Q_{n,i} \subset B(y,b R)$ for some $i$, by Corollary \ref{co:moreThanConstantLeft} and 
    Corollary \ref{cor:measofballbyINOUT},
    we have 
    \begin{align*}
        \nu(B(x,R)\cap E) \geq \nu(B(y,b R) \cap E)\geq \nu(IN(B(y,b R),n)\cap E) \\ 
        \geq c \nu(IN(B(y,b R),n)) 
        \geq \frac{c}{C_{\nu}} \nu(B(y,b R)) \geq \frac{c}{C_{\nu} C_b} \nu(B(x, 2 R)),
    \end{align*}
    where $C_b$ depends on $b$ and the doubling constant of $\nu$. Recall that according to the proof of 
    Theorem \ref{thm:maintheorem}, the measure $\nu$ is doubling on $X$.
    If on the other hand $y \in Q_{n,i} \not\subset B(y, b R)$, it follows that $2 C_1 r_{n,i} \geq b R$.
    Also by the choice of $n$, there exists $Q_{n,j} \subset B(x,R)$, and thus 
    \begin{equation*}
        Q_{n,j} \subset B(x,R) \subset B(y, 2 R) \subset B(x_{n,i}, C_1 r_{n,i} + 2 R) 
        \subset B\left(x_{n,i},\frac{(C_1 +2 ) 2 C_1}{b} r_{n,i}\right).
    \end{equation*}
    By \ref{list:comparableradius} we thus have
    $C_3\left(\frac{(C_1 +2 ) 2 C_1}{b}\right) r_{n,j} \geq r_{n,i} \geq
    \frac{b}{2 C_1} R  $.
    Since $Q_{n,j} \subset B(x,R) \cap \cap_{i=1}^{n-1} E_{i}$ by Corollary \ref{co:moreThanConstantLeft}, we have
    \begin{equation*}
      \nu(B(x,R)\cap E) \geq \nu(Q_{n,j} \cap E) \geq c \nu(Q_{n,j}) \geq \frac{c}{C_{b}} \nu(B(x,2 R)),
    \end{equation*}
    where $C_b$ only depends on $b$, $C_1$, $C_2$, $C_3$ and the doubling constant of $\nu$.
\end{proof}

The relative plumpness above is satisfied for example in the case of the Sierpinski carpets $S_{\ii}$, 
which were mentioned earlier. 
But by distorting the construction of $S_{\ii}$ slightly, we end up with an $(\alpha_n)$-regular set that is not 
relatively plump.
The next example shows that even the Lebesgue measure does not need to be doubling as a measure on the set $E$.
\begin{example}
    Let us start with the Sierpinski carpet $S_{\ii}$ defined by a sequence $\ii=(\alpha_n) $ such that 
    $\sum_{n=1}^{\infty} \alpha_n^p < 0, \forall p>0$ (here $\alpha_n$ are reciprocals of odd integers).
    Notice now that by Theorem \ref{thm:maintheorem} all the 
    doubling measures on the space $X$ give positive measure to the set $E$. (Notice that if we set $\rho = 0$ when 
    constructing the measure $\nu$, we end up with the original Lebesgue measure.) Let us now distort the cubes slightly.
    From each level, let us choose one of the removed cubes $\hat{Q}_{n,j}$ and the cube $Q_{n,k}$ right next to
    it, as in Figure \ref{pic:alteredconstruction}. By changing the radius $r_{n,j}$ to $\frac{1}{3} r_{n,j}$, 
    if necessary, and $C_1$ to $3 C_1$, we can now have an $(\alpha_n)$-regular set $E$
    such that the cube $Q_{n+1,i} \subset Q_{n,k}$.
\begin{figure}[h]
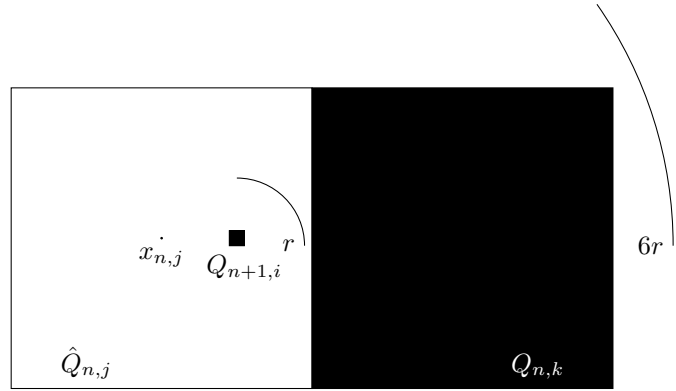

\tikzpicture
\filldraw (-2,0) circle (.01);
\node at (-2,0)[below] {$x_{n,j}$};
\draw (-4,-2) rectangle (0,2) ;
\filldraw (-1.1,-.1) rectangle (-.9,.1) ;
\draw (-.1,-.1) arc (0:90:.9);
\node at (-.1,-.1) [left] {$r$};
\draw (4.8,-.1) arc (0:35:5.6);
\node at (4.8,-.1) [left] {$6 r$};
\node at (-.9,-.1)[below] {$Q_{n+1,i}$};
\node at (-3,-2)[above] {$\hat{Q}_{n,j}$};
\filldraw (0,-2) rectangle (4,2) ;
\node at (3,-2)[above, text=white] {$Q_{n,k}$};
\endtikzpicture
\caption{Measure not doubling on $E$}
\label{pic:alteredconstruction}
\end{figure}
    It now follows that the radius $r_{n+1,i}$ of $Q_{n+1,i}$ is comparable to $ \alpha_n r_{n,j}$, 
    but the distance of $Q_{n+1,i}$ to the other components of $E_n$
    is comparable to $r_{n,j}$. If $\nu$ is now any of the measures constructed in the previous sections, we see  
    that we can have a ball $B(x,r)$ centered at $E\cap Q_{n+1,i}$ (see Figure \ref{pic:alteredconstruction})
    so that $ \nu(B(x,r) \cap E) \leq \nu(Q_{n+1,i})$,
    but $\nu(B(x,6 r) \cap E) \geq c \nu(Q_{n,k})$ by Corollary \ref{co:moreThanConstantLeft}.
    It now follows with the doubling property of $\nu$ on $[0,1]^2$ that there exist constants 
    $\tilde{C}>0$, $\hat{C}>0$ and $\lambda>0$ such that
    \begin{equation*}
        \frac{\nu(B(x,r)\cap E)}{\nu(B(x,6 r) \cap E)} \leq \frac{\nu(B(x_{n+1,i},C_1 r_{n+1,i}))}{\nu(B(x_{n+1,i},\tilde{C} r_{n,j}))} \leq \hat{C} \alpha_n^{\lambda}.   
    \end{equation*}
    Since $\alpha_n \to 0$ as $n \to \infty$ 
    and we can do this distortion at all levels,
    the measure $\nu$ restricted to the set $E$ cannot be doubling as a measure on the set $E$.
    \label{ex:nondoubling}

\end{example}
\subsection*{Acknowledgements}
    I would like to thank my advisor Ville Suomala for his comments and guidance.
    I am also most grateful for being supported by the Vilho, Yrj\"o and Kalle V\"ais\"al\"a Foundation. 

\bibliographystyle{alpha}

\bibliography{cubes}

\end{document}